\documentclass[11pt, a4paper, DIV12]{scrartcl}

\usepackage[utf8]{inputenc}
\usepackage[T1]{fontenc}
\usepackage{lmodern}
\usepackage[draft=false, kerning=true]{microtype}

\usepackage[english]{babel}
\usepackage{amsmath, amsfonts, amsthm, amssymb}
\usepackage{mathtools}
\usepackage[dvipsnames]{xcolor}
\usepackage{physics}
\usepackage{diagbox}
\usepackage{multirow}
\usepackage{refcount}
\usepackage{tablefootnote}
\usepackage[title,page]{appendix}
\usepackage{tikz}
\usetikzlibrary{decorations.pathreplacing}

\usepackage{enumitem}
\usepackage{todonotes}
\usepackage{makecell}

\newcommand{\N}{\mathbb{N}}
\newcommand{\Z}{\mathbb{Z}}

\DeclarePairedDelimiter\set\{\}
\DeclarePairedDelimiterX{\setm}[2]{\{}{\}}{#1\,\delimsize\vert\,\mathopen{}#2}
\let\abs\relax
\DeclarePairedDelimiter{\abs}{\lvert}{\rvert}
\newcommand{\gray}[1]{\textcolor{gray}{#1}}

\theoremstyle{definition}
\newtheorem{definition}{Definition}[section]

\newtheorem{conjecture}[definition]{Conjecture}
\theoremstyle{plain}
\newtheorem{theorem}[definition]{Theorem}

\numberwithin{equation}{section}

\makeatletter
\def\blfootnote{\xdef\@thefnmark{}\@footnotetext}
\makeatother

\usepackage{tikz}
\usetikzlibrary{calc}
\newcommand{\tikzmark}[1]{\tikz[overlay,remember picture] \node (#1) {};}
\newcommand{\DrawBox}[3][]{%
  \tikz[overlay,remember picture]{
    \draw[#1]
    ($(#2)+(-0.20em,0.55em)$) rectangle
    ($(#3)+(0.25em,-0.1em)$);}
}
\newcommand{\DrawLine}[3][]{%
  \tikz[overlay,remember picture]{
    \draw[#1]
    ($(#2)+(0.185em,0.55em)$) --
    ($(#3)+(0.185em,-0.1em)$);}
}
\newcommand{\DrawLineCorr}[3][]{%
  \tikz[overlay,remember picture]{
    \draw[#1]
    ($(#2)+(0.48em,0.55em)$) --
    ($(#3)+(0.185em,-0.1em)$);}
}

\usepackage{hyperref}
\hypersetup{
  a4paper,
  bookmarksopen=true,
  bookmarksnumbered=true,
  pdftitle={Large Sets without Arithmetic Progressions},
  pdfauthor={Christian Elsholtz, Benjamin Klahn, Gabriel F. Lipnik},
  pdflang=english,
  colorlinks = false,
  linkbordercolor = {MidnightBlue},
  urlbordercolor = {MidnightBlue},
  citecolor = {ForestGreen},
  citebordercolor = {ForestGreen},
  plainpages=false
}

\begin{document}

\title{Large Subsets of~$\Z_m^n$\\ without Arithmetic Progressions}
\author{Christian Elsholtz, Benjamin Klahn and Gabriel F.\ Lipnik}
\date{}

\maketitle

\begin{abstract}
  \textbf{\noindent\sffamily Abstract.} For integers $m$ and $n$, we study the problem of finding good lower bounds for the size of progression-free sets in $(\mathbb{Z}_{m}^{n},+)$. Let $r_{k}(\mathbb{Z}_{m}^{n})$ denote the maximal size of a subset of $\mathbb{Z}_{m}^{n}$ without arithmetic progressions of length~$k$ and let $P^{-}(m)$ denote the least prime factor of~$m$. We construct explicit progression-free sets 
  and obtain the following improved lower bounds for $r_{k}(\mathbb{Z}_{m}^{n})$:
  \begin{itemize}
      \item If $k\geq 5$ is odd and $P^{-}(m)\geq (k+2)/2$, then 
  \[r_k(\Z_m^n) \gg_{m,k} \frac{\bigl\lfloor \frac{k-1}{k+1}m  +1\bigr\rfloor^{n}}{n^{\lfloor \frac{k-1}{k+1}m \rfloor/2}}. \]
  \item If $k\geq 4$ is even, $P^{-}(m) \geq k$ and $m \equiv -1 \bmod k$, then
  \[r_{k}(\mathbb{Z}_{m}^{n}) \gg_{m,k} \frac{\bigl\lfloor \frac{k-2}{k}m + 2\bigr\rfloor^{n}}{n^{\lfloor \frac{k-2}{k}m + 1\rfloor/2}}.\]
  \end{itemize} 
  Moreover,
  we give some further improved lower bounds on $r_k(\Z_p^n)$ for primes $p \leq 31$ and progression lengths $4 \leq k \leq 8$.
\end{abstract}

\vspace{-1em}
\blfootnote{
  \begin{description}
    \vspace*{-1em}
    \item [Christian Elsholtz]
          \texttt{elsholtz@math.tugraz.at},
          \url{www.math.tugraz.at/~elsholtz},
          Graz University of Technology, Austria
    \item [Benjamin Klahn]
          \href{mailto:klahn@math.tugraz.at}{\texttt{klahn@math.tugraz.at}},
          Graz University of Technology, Austria
    \item [Gabriel F.\ Lipnik]
          \href{mailto:math@gabriellipnik.at}{\texttt{math@gabriellipnik.at}}, \url{www.gabriellipnik.at},
          Graz University of Technology, Austria
    \item [2020 Mathematics Subject Classification] 11B25, 05D05, 20K01
    \item [Key words and phrases] arithmetic progressions, progression-free sets, Behrend-type construction
  \end{description}}

\newpage

\section{Introduction and Main Result}

In additive combinatorics, it has been of great interest to find large subsets of $\Z_m^n \coloneqq (\Z/m\Z)^n$ without arithmetic progressions of a given length~$k$. We let $r_{k}(\mathbb{Z}_{m}^{n})$ denote the maximal size of a subset of $\mathbb{Z}_{m}^{n}$ without arithmetic progressions of length~$k$ and $P^{-}(m)$ the least prime factor of~$m$. The case $n=1$ and $k=3$ is closely related to progression-free sets in the integers; see the results by Behrend~\cite{behrend:1946:behrend-construction}, Roth~\cite{Roth:1953} and Szemer\'{e}di~\cite{szemeredi:1975:integer-sets-without-arithmetic-progressions}. The case $k=3$ and $m$ prime is strongly connected to the well-studied case of capsets~\cite{edel:2004:product-caps, ellenberg-gijswijt:2017:subsets-without-3-term-arithmetic-progression, elsholtz-lipnik:2020:caps}.
Nevertheless, there is not much literature on lower bounds on these progression-free sets, not even for primes~$m$ and general progression length~$k$. For $k=3$ the best lower bound is due to Elsholtz and Pach \cite{elsholtz-pach:2019:progression-free-sets-caps} who
adapted Behrend's method to higher dimensions, showing that there is a positive constant $C_{m}$ such that
        \begin{equation*}
          r_{3}(\mathbb{Z}_{m}^{n}) \geq
          \begin{cases}
            \frac{C_{m}}{\sqrt{n}}\bigl( \frac{m+1}{2} \bigr)^{n} & \text{if $m$ is odd,}  \\
            \frac{C_{m}}{\sqrt{n}}\bigl( \frac{m+2}{2} \bigr)^{n} & \text{if $m$ is even.}
          \end{cases}
        \end{equation*}
For $k \geq 4$, the best lower bound is due to Lin and Wolf \cite{lin-wolf:2010:sets-without-k-term-progresions}, who showed that if~$p$ is a prime and $k\leq p$, then we have
\begin{equation*}
    r_{k}(\Z_{p}^{n}) \geq \bigl(p^{2(k-1)} + p^{k-1} - 1\bigr)^{\frac{n}{2k}}.
\end{equation*}

In this paper, we adapt and extend the method of Elsholtz and Pach \cite{elsholtz-pach:2019:progression-free-sets-caps} to improve the lower bound of Lin and Wolf. We prove the following results.

\begin{theorem}
\label{thm:main-1}
    Let $m$ be an integer and let $k \geq 5$ be an odd integer. 
    Let $P^{-}(m)$ denote the least prime factor of $m$.
    If $P^{-}(m)\geq (k+2)/2$, then the following estimate holds:
    \begin{equation*}
        r_{k}(\mathbb{Z}_{m}^{n}) \gg_{m,k} \frac{\bigl\lfloor \frac{k-1}{k+1}m +1\bigr\rfloor^{n}}{
        n^{\lfloor \frac{k-1}{k+1}m \rfloor/2}}.
    \end{equation*}
\end{theorem}

\begin{theorem}
\label{thm:main-2}
 Let $k \geq 4$ be an even integer. Let $m \equiv -1 \bmod k$ and assume that $P^{-}(m) \geq k$, then we have
\[
r_{k}(\mathbb{Z}_{m}^{n}) \gg_{m,k} \frac{\bigl\lfloor \frac{k-2}{k}m  + 2\bigr\rfloor^{n}}{n^{\lfloor \frac{k-2}{k}m + 1\rfloor/2}}.  
\]  
\end{theorem}

Note that when $m = p > k$ is a prime, then Theorem \ref{thm:main-1} and Theorem \ref{thm:main-2} improve the base $p^{(k-1)/k}$ of Lin and Wolf~\cite{lin-wolf:2010:sets-without-k-term-progresions} to $\alpha_{k}p$ for some $\alpha_{k}>0$.

Moreover, our main concern in the above two theorems has been to increase the exponential numerator. It seems possible that the methods of~\cite{elsholtz-lipnik:2020:caps, elsholtz-pach:2019:progression-free-sets-caps} could also improve the polynomial denominator.

\section{Construction of Large Sets without Arithmetic Progressions}
\label{sec:results}

The work of Elsholtz and Pach~\cite{elsholtz-pach:2019:progression-free-sets-caps} suggests that for the construction of large subsets of $\mathbb{Z}_{m}^{n}$ without arithmetic progressions of length~$k$, it is a good idea to consider vectors whose entries only take values from a  prescribed set of digits. To be more precise, we consider the following sets.
\begin{definition}
  Let $D = \{d_1, \ldots, d_{\abs{D}}\} \subseteq \mathbb{Z}_m$ be a set of digits and let $n$ be an integer with $|D| \mid n$. Then we define
  \begin{equation*}
  S(D,n) \coloneqq \bigl\{ (v_{1},\ldots,v_{n}) \in \mathbb{Z}_m^{n} \mid \forall i \leq \abs{D}\colon v_{j}=d_{i}~\text{for}~n/|D|~\text{values of } j \bigr\}.
\end{equation*}
\end{definition}

Thus, $S(D,n)$ is the set of $n$-dimensional vectors for which every digit of $D$ occurs the same number of times.
The task is then to construct ``good'' sets $D \subseteq \mathbb{Z}_{m}$ such that the set $S(D,n)$ does not contain an arithmetic progression of a given length.
If~$S(D,n)$ does not contain an arithmetic progression, then we say that $D$ is \emph{admissible}.

By Stirling's formula, one can give the asymptotic lower bound
\begin{equation} 
    \label{eq:size-of-S}
  \abs{S(D,n)} = \binom{n}{\frac{n}{\abs{D}}, \ldots, \frac{n}{\abs{D}}} \gg \frac{\abs{D}^{n}}{n^{(\abs{D}-1)/2}}  
\end{equation}
as $n\to\infty$.

It remains to find large digit sets~$D$ such that~$S(D,n)$ is progression-free. In~\cite{elsholtz-pach:2019:progression-free-sets-caps} it has been shown that one can take $D=\{0, \ldots , (p-1)/2\}$
of size $\abs{D}=(p+1)/2$, without having an arithmetic progression of length $3$ in~$S(D,n)$. For odd $k \geq 5$ we shall see that we can extend this interval without having arithmetic progressions in $S(D,n)$ of length $k$.

\begin{theorem} \label{kodd}
  Let $m$ be an integer and let the progression length $k\geq 5$ be odd. If $P^{-}(m) \geq (k+2)/2$ and $n$ is an integer divisible by $\lfloor \frac{k-1}{k+1}m \rfloor +1$, then the set $S(D, n)$ with
  \begin{equation*}
  D=\set[\Big]{0,1,\ldots,\Bigl\lfloor \frac{k-1}{k+1}m \Bigr\rfloor}
  \end{equation*}
  does not contain any arithmetic progression of length $k$.
\end{theorem}
For $k=2\ell \geq 4$ even and $m \equiv -1 \bmod k$ with $P^{-}(m) \geq (k+1)/2$, we can extend the set~$D$ found in the case $k=2\ell-1$ by one element.

\begin{theorem} \label{keven}
Let $k \geq 4$ be an even integer and let  $m$ be an integer with $m \equiv -1 \bmod k$. If $P^{-}(m) \geq k$ and $n$ is an integer divisible by $\lfloor \frac{k-2}{k}m \rfloor + 2$, then the set $S(D, n)$ with
\begin{equation*}
D=\set[\Big]{0,1,\dots, \Bigl\lfloor \frac{k-2}{k}m \Bigr\rfloor, \frac{(k-1)m-1}{k}}
\end{equation*}
does not contain any arithmetic progression of length~$k$.
\end{theorem}

From our computations (see~\ref{sec:appendix-1}) it seems likely that it should also be possible to extend the construction from Theorem \ref{kodd} to integers with $m \not \equiv -1 \bmod k$.
In particular, based on experiments with small primes, we have the following conjecture.
\begin{conjecture}
  Let $p \geq 13$ be a prime with $p \equiv 1 \bmod 4$ and let $n$ be an integer divisible by~$\frac{p+3}{2}$. Then the set
  \begin{equation*}
  D=\set[\Big]{ 0,1,\ldots,\frac{p-1}{2},\frac{p+3}{2}}
  \end{equation*}
  does not yield any arithmetic progression of length~4 in $S(D,n)$.
\end{conjecture}

Finally, Table~\ref{tab:special-values} provides explicit results for some values of~$p$ respectively~$k$. As the computational effort of finding large admissible digit sets grows for increasing~$p$ and~$k$ (see Section~\ref{sec:remarks}), the values of~$p$ and~$k$ given here are rather small. In particular, we list the size of the largest admissible digit set for each pair~$(p,k)$, or a lower bound for it if the existence of larger digit sets cannot be excluded.

As an example we give a detailed discussion of the case $p=11$ and $k=3$ in \ref{app:example}. A corresponding admissible digit set of maximal cardinality as well as the number of maximal admissible digit sets for each pair~$(p,k)$  can be found in~\ref{sec:appendix-1}.

\begin{table}[htbp]
    \centering
    \begin{tabular}{c|ccccccc}
        \diagbox[width=1cm, height=0.75cm]{$p$}{$k$} & 3 & 4$\ $ & 5$\ $ & 6$\ $ & 7$\ $ & 8$\ $ \\\hline
        5 & 3 & 3 \phantom{\gray{(5)}} & 4 \gray{(4)} & 5 \phantom{\gray{(5)}}  & 5 \phantom{\gray{(5)}} & 5 \phantom{\gray{(5)}}\\
        7 & 4 & 5 \gray{(5)} & 5 \gray{(5)} & 5 \phantom{\gray{(5)}} & 6 \gray{(6)} & 7 \phantom{\gray{(5)}}\\
        11 & 6 & 7 \gray{(7)} & 8 \gray{(8)} & 9 \gray{(9)} & 9 \gray{(9)} & 9 \phantom{\gray{(5)}}\\
        13 & 7 & 8 \phantom{\gray{(5)}} & 10 \gray{(9)} & 11 \phantom{\gray{(11)}} & 11 \gray{(10)} & 11 \phantom{\gray{(11)}}\\
        17 & 9 & 10 \phantom{\gray{(11)}} & 13 \gray{(12)} & 13 \gray{(13)} & 15 \gray{(13)} & 15 \phantom{\gray{(15)}}\\
        19 & 10 & 11 \gray{(11)} & 14 \gray{(13)} & 15 \phantom{\gray{(13)}} & 16 \gray{(15)} & 17 \phantom{\gray{(17)}}\\
        23 & $\geq 12$ & $\geq 13$ \gray{(13)} & $\geq 17$ \gray{(16)} & $18$ \gray{(17)} & $\geq 19$ \gray{(18)} & $\geq 20$ \gray{(19)}\\
        29 & $\geq 15$ & $\geq 17$ \phantom{\gray{(13)}} & $\geq 21$ \gray{(20)} & $\geq 22$ \gray{(21)} & $\geq 24$ \gray{(22)} & $\geq 25$ \phantom{\gray{(25)}}\\
        31 & $\geq 16$ & $\geq 18$ \gray{(17)} & $\geq 22$ \gray{(21)} & $\geq 23$ \phantom{\gray{(21)}} & $\geq 26$ \gray{(24)} & $\geq 26$ \gray{(25)}
\end{tabular}
    \caption{Maximal size of digit sets~$D$ modulo~$p$ such that $S(D,n)$ does not contain an arithmetic progression of length~$k$. The numbers given in parentheses are the bounds that we obtain from the general Theorems~\ref{keven} and~\ref{kodd}.}
    \label{tab:special-values}
\end{table}

\section{Proofs}
\label{sec:proofs}

\begin{proof}[Proof of Theorem~\ref{kodd}]
  Assume for the sake of a contradiction that $S(D,n)$ contains a non-constant arithmetic progression $\mathbf{v}_{1},\mathbf{v}_{2}$, \ldots, $\mathbf{v}_{k}$. Denote the $j$th coordinate of $\mathbf{v}_{i}$ by $\mathbf{v}_{i}^{(j)}$, i.e., $\mathbf{v}_{i}=(\mathbf{v}_{i}^{(1)},\mathbf{v}_{i}^{(2)}, \dots, \mathbf{v}_{i}^{(k)})$. Notice that for every $1 \leq j \leq n$ the elements $\mathbf{v}_{1}^{(j)},\mathbf{v}_{2}^{(j)},\dots, \mathbf{v}_{k}^{(j)}$ form an arithmetic progressions in~$\mathbb{Z}_{m}$.

  Let $d \leq \frac{m(k-1)}{k+1}$ be the largest integer such that $d$ occurs in a non-trivial progression $(\mathbf{v}_{i}^{(j_{0})})_{1 \leq i \leq k}$. As $d$ occurs equally often in every vector $\mathbf{v}_{i}$, we may choose $j=j_{0}$ such that $(\mathbf{v}_{i}^{(j_{0})})_{1 \leq i \leq k}$ is non-constant and $\mathbf{v}_{\ell+1}^{(j_{0})}=d$ with $\ell  =(k-1)/2$.

  By writing $\mathbf{v}_{i}^{(j_{0})} = \mathbf{v}_{1}^{(j_{0})}+(i-1)c$ for some non-zero element $c \in \mathbb{Z}_{m}$ we see that the elements $\mathbf{v}_{i}^{(j_{0})}$ with $1 \leq i \leq \ell +1$ are pairwise distinct since $P^{-}(m) > \ell$. By the maximality of $d$ we have $((\mathbf{v}_{i}^{(j_{0})})_{1 \leq i \leq k}) \subset [0,d]$. Thus, by the pigeonhole principle there exist indices $1 \leq i,i' \leq \ell +1$ such that
  \[
  	0<\mathbf{v}_{i}^{(j_{0})}-\mathbf{v}_{i'}^{(j_{0})} \leq \frac{d}{\ell} < \frac{m}{\ell +1},
  \] 
where the last inequality is strict since $P^{-}(m)>\ell +1$.

However, this implies that $\mathbf{v}_{\ell+i-i'}^{(j_{0})} \in \{d+1,\ldots,m-1 \}$, a contradiction.
\end{proof}

\begin{proof}[Proof of Theorem~\ref{keven}]
By Theorem 2.2 it suffices to show that a non-trivial arithmetic progression in $D$ cannot have $h:=\frac{(k-1)m-1}{k}$ in the first position. Assume that there is such a progression $v_{1},v_{2},\dots, v_{k}$ with $v_{1} = h$. Again, as $P^{-}(m) \geq k$ all elements $v_{1},v_{2},\dots, v_{k}$ are pairwise different. Note that the nearest elements in $D$ to $h$, namely the residue classes of $m$ (which is $0$) and $\lfloor \frac{(k-2)m}{k} \rfloor$, both have distance
\begin{equation}
\abs{m-h}=\abs[\Big]{h - \Bigl\lfloor \frac{(k-2)m}{k} \Bigr\rfloor} = \frac{m+1}{k}    \label{minDist}
\end{equation}
to $h$.

The elements $v_{2}$, $v_{3}$, \dots, $v_{k}$ are all different from $h$ and must therefore all lie in the interval $[0,\lfloor \frac{(k-2)m}{k} \rfloor]$. Thus, by the pigeon hole principle there are two elements $v_{i}$ and $v_{j}$ with $k \geq j>i \geq 2$ and distance $\abs{v_{i}-v_{j}}<\frac{m}{k}$. But this would mean that also $\abs{v_{1}-v_{j-i+1}}<\frac{m}{k}$, contradicting the minimal distance from $h$ to another element given in~\eqref{minDist}. Thus, there can be no such progression.
\end{proof}

\begin{proof}[Proof of Theorem~\ref{thm:main-1}]
By Theorem~\ref{kodd}  we can find a digit set~$D$ of size at least $m(k-1)/(k+1)$ such that there is no arithmetic progression of length~$k$ in~$S(D,n)$ for $n\in\N$ with $\abs{D}\mid n$. By~\eqref{eq:size-of-S}  we find
\begin{equation*}
    \abs{S(D,n)} 
    \gg_{m,k} \frac{\bigl\lfloor\frac{k-1}{k+1}m + 1\bigr\rfloor^{n}}{n^{\lfloor\frac{k-1}{k+1}m\rfloor/2}}.
\end{equation*}
If $\abs{D} \nmid n$, say $n = \abs{D}q+r$ and $1 \leq r < \abs{D}$, we can embed the set $S(D,n-r)$ into~$\mathbb{Z}_{m}^{n}$ by simply putting zeroes in the last $r$ coordinates. The image does not contain any arithmetic progressions and is also of size
\begin{equation*}
    \abs{S(D,n-r)} 
    \gg_{m,k} \frac{\bigl\lfloor \frac{k-1}{k+1}m + 1\bigr\rfloor^{n}}{n^{\lfloor\frac{k-1}{k+1}m\rfloor/2}},
\end{equation*}
as claimed.
\end{proof}

Theorem~\ref{thm:main-2} can be proven analogously, as a conclusion of Theorem~\ref{keven}.

\section{Finding Admissible Digit Sets}
\label{sec:remarks}

In this section, we present approaches to find admissible digit sets. In general, we are not able to strengthen the bounds given in Theorem~\ref{thm:main-1} and Theorem~\ref{thm:main-2}. However, we can improve the bounds for small primes by using computer power and applying the following method. See also Table~\ref{tab:special-values}.

Most of the ideas given in this section can also be found in~\cite{elsholtz-lipnik:2020:caps}, where the authors use similar techniques for capset constructions. 

\subsection{Modelling the Problem}
As already seen in the previous section, the described construction relies on finding large digit sets~$D\subseteq\Z_p$ such that $S(D,n)$ does not contain arithmetic progressions for all dimensions~$n\in\N$ with $\abs{D}\mid n$.
For this purpose, let $k\geq 3$ be the progression length and let~$p$ be a prime. Moreover, let
  $P_k(D)\subseteq D^k$
be the set of non-trivial $k$-term arithmetic progressions in~$D$.
Assume that there are $k$~points in~$S(D,n)$ which form an arithmetic progression for some $n\in\N$. For each progression~$v = (v_1,\ldots,v_k)\in P_k(D)$, let~$x_{v}$ be a variable which counts the occurrences of~$v$ in
the components of these~$k$~points. Due to the fact that each
digit~$d\in D$ has to occur the same number of time in each of the $k$~points, the equation
\begin{equation}
  \label{eq:equations-for-progressions}
  \sum_{\substack{v\in P_k(D)\\ v_{i} = d}}x_{v} = \sum_{\substack{v\in P_k(D)\\ v_{j} = d}}x_{v}
\end{equation}
has to hold for each digit~$d\in D$ and for any pair~$(i,j)$ with $1\leq i < j \leq k$.

It is easy to check that the non-existence
of a non-negative non-trivial integral solution
$(x_{v}\mid v\in P_k(D))$ of the system of equations given in~\eqref{eq:equations-for-progressions} is equivalent to
the non-existence of a $k$-term arithmetic progression in~$S(D,n)$.
Hence, if we want to prove the admissibility of some digit set~$D$, we have to
ensure that the set $\mathcal{P}(D) = \setm{x\in\Z_{\geq 0}^{\ell}}{Ax=0}$
only contains the zero vector, where the system of linear
equations~$Ax = 0$ describes the equations stated
in~\eqref{eq:equations-for-progressions}. If we want to show that a digit set is not admissible, on the other hand, then we have to find a non-negative non-trivial solution of $Ax = 0$. This solution directly corresponds to a $k$-term arithmetic progression in~$S(D,n)$, for infinitely many dimensions~$n$.

Both can be achieved by methods
of integer linear programming. Unfortunately, the problem of deciding if a
polyhedron contains an integral point is computationally hard and in general NP-complete~\cite{garey-johnson-computers-and-intractability}. This indicates
that checking this condition for all possible digit sets modulo~$p$
can only be done for small~$p$---and has been done for primes $5 \leq p \leq 31$ and progression length $3\leq k \leq 8$, as Table~\ref{tab:special-values} indicates.

\subsection{Reducibility as a Sufficient Condition}
\label{sec:matrix-red}
Next, we briefly describe a technique which allows showing admissibility in a computationally simple and comprehensible way.
For some fixed digit set~$D$ and the corresponding constraint matrix~$A$ as mentioned above, let~$B$ be a fixed matrix which is equivalent to~$A$ in the sense that there exists an invertible matrix~$T$ such that $TA = B$. This certainly implies $\setm{x\in\Z_{\geq 0}^{\ell}}{Ax=0} = \mathcal{P}(D) = \setm{x\in\Z_{\geq 0}^{\ell}}{Bx=0}$.

Remember that we want to show the emptiness of $\mathcal{P}(D)$.
Therefore, if some non-zero row~$i$ of~$B$ only contains non-negative or non-positive entries, then it clearly follows that the variables corresponding to non-zero entries of this row have to be zero. This is because we are looking for non-negative solutions~$x$ of $B x = 0$, and if the said variables were non-zero, then the equation corresponding to row~$i$ of~$B$ would not have such a solution.

Consequently, we remove those columns of~$B$ which belong to these variables, i.e., columns of~$B$ with non-zero entry in row~$i$, and then proceed with the remaining matrix and the next non-negative or non-positive row. The deletion of columns possibly brings out new non-negative or non-positive rows. The process determines if no non-negative or non-positive non-zero row is left in the matrix. If at the end all columns of~$B$ are deleted---which means that all variables~$x_i$ have to be zero and that this is the only non-negative integral solution---, then the digit set~$D$ is admissible.

Two very natural choices\footnote{Note the approach of reducing the set~$P_k(D)$ due to certain conditions on the occurring digits which is used in~\cite{elsholtz-lipnik:2020:caps, elsholtz-pach:2019:progression-free-sets-caps} is a special case of the reducibility presented here, namely with initial matrix~$B = A$. For further details we refer to the mentioned papers.} for the initial matrix~$B$ are $B = A$ and $B = A_{\mathrm{ech}}$, where $A_{\mathrm{ech}}$ denotes the reduced row echelon form of~$A$. It turns out that these choices are not only intuitive but also very successful and good enough for our purpose: We were able to verify 50 of the 54 bounds given in Table~\ref{tab:special-values} using them; see~\ref{sec:appendix-1}. One comprehensible example with a different choice of~$B$ can be found in \ref{app:example}.

\section*{Data Availability and Acknowledgements}

All datasets which are generated and analysed during the current study and which are not included in this work are available from the corresponding author on reasonable request.

The authors acknowledge the support of the Austrian Science Fund (FWF): W\,1230 and I\,4945-N. Moreover, the authors thank the referees for their helpful comments.

\bibliography{literature}

\providecommand{\bysame}{\leavevmode\hbox to3em{\hrulefill}\thinspace}
\providecommand{\MR}{\relax\ifhmode\unskip\space\fi MR }
\providecommand{\MRhref}[2]{%
  \href{http://www.ams.org/mathscinet-getitem?mr=#1}{#2}
}
\providecommand{\href}[2]{#2}
\begin{thebibliography}{1}

\bibitem{behrend:1946:behrend-construction}
Felix~A. Behrend, \href{http://dx.doi.org/10.1073/pnas.32.12.331}{\emph{On sets
  of integers which contain no three terms in arithmetical progression}}, Proc.
  Nat. Acad. Sci. USA \textbf{32} (1946), 331--332. \MR{18694}

\bibitem{edel:2004:product-caps}
Yves Edel, \href{http://dx.doi.org/10.1023/A:1027365901231}{\emph{Extensions of
  generalized product caps}}, Des. Codes Cryptogr. \textbf{31} (2004), no.~1,
  5--14. \MR{2031694}

\bibitem{ellenberg-gijswijt:2017:subsets-without-3-term-arithmetic-progression}
Jordan~S. Ellenberg and Dion Gijswijt,
  \href{http://dx.doi.org/10.4007/annals.2017.185.1.8}{\emph{On large subsets
  of {$\Bbb F^n_q$} with no three-term arithmetic progression}}, Ann. of Math.
  (2) \textbf{185} (2017), no.~1, 339--343. \MR{3583358}

\bibitem{elsholtz-lipnik:2020:caps}
Christian Elsholtz and Gabriel~F. Lipnik, \emph{Exponentially larger affine and
  projective caps}, submitted, 2020.

\bibitem{elsholtz-pach:2019:progression-free-sets-caps}
Christian Elsholtz and Péter~P. Pach,
  \href{https://doi.org/10.1007/s10623-020-00769-0}{\emph{Caps and
  progression-free sets in $\mathbb{Z}_m^n$}}, Des. Codes Cryptogr. \textbf{88}
  (2020), 2133–2170.

\bibitem{garey-johnson-computers-and-intractability}
Michael~R. Garey and David~S. Johnson, \emph{Computers and intractability}, W.
  H. Freeman and Co., San Francisco, Calif., 1979, A guide to the theory of
  NP-completeness, A Series of Books in the Mathematical Sciences. \MR{519066}

\bibitem{lin-wolf:2010:sets-without-k-term-progresions}
Yuncheng Lin and Julia Wolf,
  \href{http://dx.doi.org/10.1016/j.ejc.2009.12.001}{\emph{On subsets of {$\Bbb
  F^n_q$} containing no {$k$}-term progressions}}, European J. Combin.
  \textbf{31} (2010), no.~5, 1398--1403. \MR{2644427}

\bibitem{Roth:1953}
Klaus~F. Roth, \href{http://dx.doi.org/10.1112/jlms/s1-28.1.104}{\emph{On
  certain sets of integers}}, J. London Math. Soc. \textbf{28} (1953),
  104--109. \MR{51853}

\bibitem{szemeredi:1975:integer-sets-without-arithmetic-progressions}
Endre Szemer\'{e}di, \href{http://dx.doi.org/10.4064/aa-27-1-199-245}{\emph{On
  sets of integers containing no {$k$} elements in arithmetic progression}},
  Acta Arith. \textbf{27} (1975), 199--245. \MR{369312}

\end{thebibliography}
\bibliographystyle{amsplainurl}

\begin{appendix}
\gdef\thesection{Appendix \Alph{section}}

\section{Illustration of Reducibility}
\label{app:example}
In the following, we exemplarily illustrate the concept of reducibility. For this purpose, let us have a look at the modulus $p = 11$ and the progression length $k=3$.

We show that the digit set $D_1 = \set{0, 1, 2, 3, 4, 5}$ is admissible (even though this is already known from~\cite{elsholtz-pach:2019:progression-free-sets-caps}) by deducing its reducibility with reduced row echelon form as initial matrix. The set~$P_3(D_1)$ of non-trivial $3$-term arithmetic progressions in~$D_1$ is given by
\begin{align*}
   P_3(D_1) = \bigl\{&(0, 1, 2),
 (0, 2, 4),
 (1, 2, 3),
 (1, 3, 5),
 (2, 3, 4),
 (2, 1, 0),\\
 &(3, 4, 5),
 (3, 2, 1),
 (4, 2, 0),
 (4, 3, 2),
 (5, 3, 1),
 (5, 4, 3)\bigr\},
\end{align*}
and thus, the constraint matrix~$A$ is given by
\begin{equation*}
    A = \left(\begin{smallmatrix}
1 & 1 & 0 & 0 & 0 & 0 & 0 & 0 & 0 & 0 & 0 & 0 \\
-1 & 0 & 1 & 1 & 0 & -1 & 0 & 0 & 0 & 0 & 0 & 0 \\
0 & -1 & -1 & 0 & 1 & 1 & 0 & -1 & -1 & 0 & 0 & 0 \\
0 & 0 & 0 & -1 & -1 & 0 & 1 & 1 & 0 & -1 & -1 & 0 \\
0 & 0 & 0 & 0 & 0 & 0 & -1 & 0 & 1 & 1 & 0 & -1 \\
0 & 0 & 0 & 0 & 0 & 0 & 0 & 0 & 0 & 0 & 1 & 1 \\
1 & 1 & 0 & 0 & 0 & -1 & 0 & 0 & -1 & 0 & 0 & 0 \\
0 & 0 & 1 & 1 & 0 & 0 & 0 & -1 & 0 & 0 & -1 & 0 \\
-1 & 0 & 0 & 0 & 1 & 1 & 0 & 0 & 0 & -1 & 0 & 0 \\
0 & 0 & -1 & 0 & 0 & 0 & 1 & 1 & 0 & 0 & 0 & -1 \\
0 & -1 & 0 & 0 & -1 & 0 & 0 & 0 & 1 & 1 & 0 & 0 \\
0 & 0 & 0 & -1 & 0 & 0 & -1 & 0 & 0 & 0 & 1 & 1
\end{smallmatrix}\right),
\end{equation*}
where the first six rows represent equations which arise from the first and second position in the vectors of~$P_3(D_1)$ (i.e., $i=1$ and $j=2$ in~\eqref{eq:equations-for-progressions}), and the latter six rows represent the constraints for the positions one and three in the vectors of~$P_3(D_1)$ (i.e., $i=1$ and $j=3$ in~\eqref{eq:equations-for-progressions}). 
Moreover, its reduced row echelon form is given by
\begin{equation*}
    A_{\mathrm{ech}} = \left(\begin{smallmatrix}
1 & 0 & 0 & 0 & 0 & 0 & 0 & -1 & -1 & 0 & 0 & 1 \\
0 & 1 & 0 & 0 & 0 & 0 & 0 & 1 & 1 & 0 & 0 & -1 \\
0 & 0 & 1 & 0 & 0 & 0 & 0 & -1 & -1 & -1 & 0 & 2 \\
0 & 0 & 0 & 1 & 0 & 0 & 0 & 0 & 1 & 1 & 0 & -1 \\
0 & 0 & 0 & 0 & 1 & 0 & 0 & -1 & -2 & -1 & 0 & 1 \\
0 & 0 & 0 & 0 & 0 & 1 & 0 & 0 & 1 & 0 & 0 & 0 \\
0 & 0 & 0 & 0 & 0 & 0 & 1 & 0 & -1 & -1 & 0 & 1 \\
0 & 0 & 0 & 0 & 0 & 0 & 0 & 0 & 0 & 0 & 1 & 1 \\
0 & 0 & 0 & 0 & 0 & 0 & 0 & 0 & 0 & 0 & 0 & 0 \\
0 & 0 & 0 & 0 & 0 & 0 & 0 & 0 & 0 & 0 & 0 & 0 \\
0 & 0 & 0 & 0 & 0 & 0 & 0 & 0 & 0 & 0 & 0 & 0 \\
0 & 0 & 0 & 0 & 0 & 0 & 0 & 0 & 0 & 0 & 0 & 0
\end{smallmatrix}\right).
\end{equation*}
 The described matrix reduction is given as follows, where the non-negative respectively non-positive rows as well as the columns which have to be deleted in the next step are marked:
\begin{align*}
\left(\begin{smallmatrix}
1 & 0 & 0 & 0 & 0 & \tikzmark{o1}0 & 0 & -1 & \tikzmark{o2}-1 & 0 & \tikzmark{o3}0 & \tikzmark{o4}1 \\
0 & 1 & 0 & 0 & 0 & 0 & 0 & 1 & 1 & 0 & 0 & -1 \\
0 & 0 & 1 & 0 & 0 & 0 & 0 & -1 & -1 & -1 & 0 & 2 \\
0 & 0 & 0 & 1 & 0 & 0 & 0 & 0 & 1 & 1 & 0 & -1 \\
0 & 0 & 0 & 0 & 1 & 0 & 0 & -1 & -2 & -1 & 0 & 1 \\
\tikzmark{l1}0 & 0 & 0 & 0 & 0 & 1 & 0 & 0 & 1 & 0 & 0 & 0\tikzmark{r1} \\
0 & 0 & 0 & 0 & 0 & 0 & 1 & 0 & -1 & -1 & 0 & 1 \\
\tikzmark{l2}0 & 0 & 0 & 0 & 0 & 0 & 0 & 0 & 0 & 0 & 1 & 1\tikzmark{r2} \\
0 & 0 & 0 & 0 & 0 & 0 & 0 & 0 & 0 & 0 & 0 & 0 \\
0 & 0 & 0 & 0 & 0 & 0 & 0 & 0 & 0 & 0 & 0 & 0 \\
0 & 0 & 0 & 0 & 0 & 0 & 0 & 0 & 0 & 0 & 0 & 0 \\
0 & 0 & 0 & 0 & 0 & \tikzmark{u1}0 & 0 & 0 & \tikzmark{u2}0 & 0 & \tikzmark{u3}0 & \tikzmark{u4}0
\end{smallmatrix}\right)
\DrawBox[ForestGreen, thick]{l1}{r1}
\DrawBox[ForestGreen, thick]{l2}{r2}
\DrawLine[Maroon, thick]{o1}{u1}
\DrawLineCorr[Maroon, thick]{o2}{u2}
\DrawLine[Maroon, thick]{o3}{u3}
\DrawLine[Maroon, thick]{o4}{u4}
\leadsto
\left(\begin{smallmatrix}
1 & \tikzmark{o5}0 & 0 & \tikzmark{o6}0 & 0 & 0 & \tikzmark{o7}-1 & \tikzmark{o8}0 \\
\tikzmark{l5}0 & 1 & 0 & 0 & 0 & 0 & 1 & 0\tikzmark{r5}\\
0 & 0 & 1 & 0 & 0 & 0 & -1 & -1\\
\tikzmark{l6}0 & 0 & 0 & 1 & 0 & 0 & 0  & 1\tikzmark{r6}\\
0 & 0 & 0 & 0 & 1 & 0 & -1 & -1\\
0 & 0 & 0 & 0 & 0 & 0 & 0 & 0\\
0 & 0 & 0 & 0 & 0 & 1 & 0 & -1\\
0 & 0 & 0 & 0 & 0 & 0 & 0 & 0\\
0 & 0 & 0 & 0 & 0 & 0 & 0 & 0\\
0 & 0 & 0 & 0 & 0 & 0 & 0 & 0\\
0 & 0 & 0 & 0 & 0 & 0 & 0 & 0\\
0 & \tikzmark{u5}0 & 0 & \tikzmark{u6}0 & 0 & 0 & \tikzmark{u7}0 & \tikzmark{u8}0
\end{smallmatrix}\right)
\DrawBox[ForestGreen, thick]{l5}{r5}
\DrawBox[ForestGreen, thick]{l6}{r6}
\DrawLine[Maroon, thick]{o5}{u5}
\DrawLine[Maroon, thick]{o6}{u6}
\DrawLineCorr[Maroon, thick]{o7}{u7}
\DrawLine[Maroon, thick]{o8}{u8}
\leadsto
\left(\begin{smallmatrix}
1 & 0 & 0 & 0\\
0 & 0 & 0 & 0\\
0 & 1 & 0 & 0\\
0 & 0 & 0 & 0\\
0 & 0 & 1 & 0\\
0 & 0 & 0 & 0\\
0 & 0 & 0 & 1\\
0 & 0 & 0 & 0\\
0 & 0 & 0 & 0\\
0 & 0 & 0 & 0\\
0 & 0 & 0 & 0\\
0 & 0 & 0 & 0
\end{smallmatrix}\right)
\leadsto
\left(\begin{smallmatrix}
\ \\
\ \\
\ \\
\ \\
\ \\
\ \\
\ \\
\ \\
\ \\
\ \\
\ \\
\ \\
\ \\
\
\end{smallmatrix}\right).
\end{align*}
The last step trivially follows. As a consequence, the only non-negative vector~$x$ that solves $Ax = 0$ is the zero vector. This implies that $D_1$ is reducible and thus, also admissible.

\section{Admissible Digit Sets (Verification of Table~\ref{tab:special-values})}
\label{sec:appendix-1}

In the following, we list one maximal admissible set for each pair $(p, k)$ with $5 \leq p \leq 31$ and $3 \leq k \leq 8$. Admissibility was checked via reducibility as presented in Section~\ref{sec:remarks}; we list the lexicographically first admissible digit set which is reducible with initial matrix~$A$ or~$A_{\mathrm{ech}}$, where this is possible. We also give the initial matrices with which we have established reducibility of the corresponding digit sets.

Moreover, for small primes we also give the number of maximal admissible digit sets. This result was obtained by the computational integer programming approach. Note that many admissible digit sets are in some sense symmetric to each other. We have refrained from filtering out such patterns because the given number should only convey a sense for its range.
To keep the following tables concise, we use the usual notation for discrete intervals, i.e.,
\begin{equation*}
    [a,b] \coloneqq \setm{x\in\Z}{a\leq x\leq b}
\end{equation*}
for integers $a$ and $b$ with $0\leq a \leq b <p$, and we consider these sets to be subsets of~$\Z_p$.

The admissibility of the four digit sets marked with a star (*) has been checked by using the integer programming approach which is described in Section~\ref{sec:matrix-red}, because no reducible digit set has been found of the same size, neither with initial matrix~$A$ nor with initial matrix~$A_\mathrm{ech}$. (Numerous digit sets with one element less are reducible with initial matrix~$A$ respectively~$A_\mathrm{ech}$, though.) 

\begin{table}[htbp]
    \centering
    \begin{tabular}{c|c|c|c}
        \makecell{$p$}  & \makecell{one maximal\\ admissible digit set} & \makecell{initial\\ matrix $B$} & \makecell{number of maximal\\ admissible digit sets}\\\hline
        5 & $[0,2]$ & $A$, $A_{\mathrm{ech}}$ & 10\\
        7 & $[0,3]$ & $A$, $A_{\mathrm{ech}}$ & 35\\
        11 & $[0,5]$ & $A$, $A_{\mathrm{ech}}$ & 275\\
        13 & $[0,6]$ & $A$ & 546\\
        17 & $[0,8]$ & $A$ & 1496\\
        19 & $[0,9]$ & $A$ & 2223\\
        23 & $[0,11]$ & $A$ & 4301\\
        29 & $[0,14]$ & $A$ & \\
        31 & $[0,15]$ & $A$ &\\
    \end{tabular}
    \caption{Progression length~$k=3$ (see also~\cite{elsholtz-pach:2019:progression-free-sets-caps})}
    \label{tab:k-3}
\end{table}

\begin{table}[htbp]
    \centering
    \begin{tabular}{c|c|c|c}
        \makecell{$p$}  & \makecell{one maximal\\ admissible digit set} & \makecell{initial\\ matrix $B$} & \makecell{number of maximal\\ admissible digit sets}\\\hline
        5 & $[0,2]$& $A$, $A_{\mathrm{ech}}$ & 10\\
        7 & $[0,4]$ & $A$, $A_{\mathrm{ech}}$ & 21\\
        11 & $[0,6]$ & $A_{\mathrm{ech}}$ & 220\\
        13 & $[0,6]\cup\set{8}$ & $A_{\mathrm{ech}}$ & 468\\
        17 & $[0,8]\cup\set{10}$ & $A_{\mathrm{ech}}$ & 5848\\
        19 & $[0,10]$ & $A_{\mathrm{ech}}$ & 16416\\
        23 & $[0,12]$ & $A_{\mathrm{ech}}$ & \\
        29 & $[0,12]\cup\set{14,25,27,28}$ & $A_{\mathrm{ech}}$ & \\
        31 & $[0,12]\cup\set{14,16,27,29,30}$ & $A_{\mathrm{ech}}$ & \\
    \end{tabular}
    \caption{Progression length~$k=4$}
    \label{tab:k-4}
\end{table}

\begin{table}[htbp]
    \centering
    \begin{tabular}{c|c|c|c}
        \makecell{$p$}  & \makecell{one maximal\\ admissible digit set} & \makecell{initial\\ matrix $B$} & \makecell{number of maximal\\ admissible digit sets}\\\hline
        5 & $[0,3]$ & $A$, $A_{\mathrm{ech}}$ & 5\\
        7 & $[0,4]$ & $A$, $A_{\mathrm{ech}}$ & 21\\
        11 & $[0,7]$ & $A$, $A_{\mathrm{ech}}$ & 165\\
        13 & $[0,9]$ & $A_{\mathrm{ech}}$ & 286\\
        17 & $[0,9]\cup[11,13]$ & $A$ & 1768\\
        19 & $[0,13]$ & $A_{\mathrm{ech}}$ & 10089\\
        23 & $[0,12]\cup[14,16]\cup\set{18}$ & $A$ & \\
        29 & $[0,15]\cup[17,20]\cup\set{26}$ & $A$ & \\
        31 & $[0,17]\cup\set{19,20,26,29}$ & $A$ & \\
    \end{tabular}
    \caption{Progression length~$k=5$}
    \label{tab:k-5}
\end{table}
\begin{table}[htbp]
    \centering
    \begin{tabular}{c|c|c|c}
        \makecell{$p$}  & \makecell{one maximal\\ admissible digit set} & \makecell{initial\\ matrix $B$} & \makecell{number of maximal\\ admissible digit sets}\\\hline
        5 & $[0,4]$ & $A$, $A_{\mathrm{ech}}$ & 1\\
        7 & $[0,4]$ & $A$, $A_{\mathrm{ech}}$ & 21\\
        11 & $[0,8]$ & $A$, $A_{\mathrm{ech}}$ & 55\\
        13 & $[0,10]$ & $A_{\mathrm{ech}}$ & 78\\
        17 & $[0,12]$ & $A_{\mathrm{ech}}$ & 2312 \\
        19 & $[0,14]$ & $A_{\mathrm{ech}}$ & 2052\\
        23 & $[0,13]\cup\set{15,19,21,22}$ & $A_{\mathrm{ech}}$ & 23529\\
        29 & $[0,15]\cup\set{17,18,23}\cup[25,17]$ & $A_{\mathrm{ech}}$ & \\
        31 & $[0,18]\cup\set{20,26,29,30}$ & $A_{\mathrm{ech}}$ & \\
    \end{tabular}
    \caption{Progression length~$k=6$}
    \label{tab:k-6}
\end{table}

\begin{table}[htbp]
    \centering
    \begin{tabular}{c|c|c|c}
        \makecell{$p$}  & \makecell{one maximal\\ admissible digit set} & \makecell{initial\\ matrix $B$} & \makecell{number of maximal\\ admissible digit sets}\\\hline
        5 & $[0,4]$ & $A$, $A_{\mathrm{ech}}$ & 1 \\
        7 & $[0,5]$ & $A$, $A_{\mathrm{ech}}$ & 7\\
        11 & $[0,8]$ & $A$, $A_{\mathrm{ech}}$ & 55\\
        13 & $[0,10]$ & $A$, $A_{\mathrm{ech}}$ & 78\\
        17 & $[0,14]^*$ &  & 136\\
        19 & $[0,15]$ & $A_{\mathrm{ech}}$ & 969\\
        23 & $[0,18]$ & $A_{\mathrm{ech}}$ & \\
        29 & $[0,23]^*$ &  & \\
        31 & $[0,25]^*$ &  & \\
    \end{tabular}
    \caption{Progression length~$k=7$}
    \label{tab:k-7}
\end{table}

\begin{table}[htbp]
    \centering
    \begin{tabular}{c|c|c|c}
        \makecell{$p$}  & \makecell{one maximal\\ admissible digit set} & \makecell{initial\\ matrix $B$} & \makecell{number of maximal\\ admissible digit sets}\\\hline
        5 & $[0,4]$ & $A$, $A_{\mathrm{ech}}$ & 1\\
        7 & $[0,6]$ & $A$, $A_{\mathrm{ech}}$ & 1\\
        11 & $[0,8]$ & $A$, $A_{\mathrm{ech}}$ & 55\\
        13 & $[0,10]$ & $A$, $A_{\mathrm{ech}}$ & 78 \\
        17 & $[0,14]$ & $A_{\mathrm{ech}}$ & 136\\
        19 & $[0,14]\cup\set{16,17}$ & $A_{\mathrm{ech}}$ & 171\\
        23 & $[0,15]\cup\set{18, 19, 21, 22}$ & $A_{\mathrm{ech}}$ & 1771\\
        29 & $[0,24]^*$ &  & \\
        31 & $[0,19]\cup\set{22,24,25}\cup[28,30]$ & $A$ & \\
    \end{tabular}
    \caption{Progression length~$k=8$}
    \label{tab:k-8}
\end{table}

\end{appendix}

\end{document}